\documentclass{article}
\usepackage[margin=1in, bottom=1in, top=1in, marginparwidth=60pt]{geometry} %1 inch margins
\usepackage{amsmath, amssymb, amstext}
\usepackage{fancyhdr}
\usepackage{mathtools}
\usepackage[ruled,vlined]{algorithm2e}
\usepackage{algpseudocode}
\usepackage{theorem}
\usepackage{tikz}
\usepackage{authblk}
\usepackage[colorlinks]{hyperref}
\usepackage{enumerate}  
\usepackage{bbm}
\usepackage{mathbbol}
\usepackage{comment}
\usepackage{subfig}
\usepackage{multirow}
\usepackage{numprint}
\usepackage{wrapfig}
\usepackage{booktabs}
 \newcommand{\R}{\mathbb{R}}
 \newcommand{\Z}{\mathbb{Z}}
\newcommand{\Q}{\mathbb{Q}}

\DeclareMathOperator{\uncovOp}{uncov}
\newcommand{\uncov}{\uncovOp}

\newcommand{\size}[1]{\ensuremath{\left|#1\right|}}

\newcommand\restr[2]{{% we make the whole thing an ordinary symbol
  \left.\kern-\nulldelimiterspace % automatically resize the bar with \right
  #1 % the function
  \vphantom{\big|} % pretend it's a little taller at normal size
  \right|_{#2} % this is the delimiter
  }}

\newenvironment{proof}{{\bf Proof:  }}{\hfill\rule{2mm}{2mm}}

\numberwithin{equation}{section} % Number equations within sections (i.e. 1.1, 1.2, 2.1, 2.2 instead of 1, 2, 3, 4)
\numberwithin{figure}{section} % Number figures within sections (i.e. 1.1, 1.2, 2.1, 2.2 instead of 1, 2, 3, 4)
\numberwithin{table}{section} % Number tables within sections (i.e. 1.1, 1.2, 2.1, 2.2 instead of 1, 2, 3, 4)

\newtheorem{fact}{Fact}[section]
\newtheorem{lemma}[fact]{Lemma}
\newtheorem{theorem}[fact]{Theorem}

\newtheorem{corollary}[fact]{Corollary}

\begin{document}
\title{Optimized Cranial Bandeau Remodeling \thanks{The authors were supported by the NSERC Discovery Grant program and acknowledge the support of the Natural Sciences and Engineering Research Council of Canada (NSERC).}}
\author[1]{James Drake}
\author[2]{Marina Drygala}
\author[2]{Ricardo Fukasawa} 
\author[2]{Jochen Koenemann}
\author[4]{Andr\'e Linhares}
\author[1]{Thomas Looi}  
\author[1]{John Phillips}  
\author[2]{David Qian} 
\author[3]{Nikoo Saber} 
\author[1,*]{Justin Toth} 
\author[1]{Chris Woodbeck} 
\author[1]{Jessie Yeung}

\affil[1]{The Hospital for Sick Children, Toronto, Ontario} 
\affil[2]{University of Waterloo, Waterloo, Ontario}
\affil[3]{MyJiko, Toronto, Ontario}
\affil[4]{University of Waterloo, Waterloo, Ontario. Now at Google Research}
\affil[*]{Corresponding Author: Justin Toth, \href{mailto:wjtoth@uwaterloo.ca}{wjtoth@uwaterloo.ca}}
\setcounter{Maxaffil}{0}
\renewcommand\Affilfont{\itshape\small}
%
% First names are abbreviated in the running head

% \email{\{jochen,,wjtoth\}@uwaterloo.ca}
%

\maketitle              % typeset the header of the contribution
\begin{abstract}
    {\em Craniosynostosis}, a condition affecting $1$ in $2000$
    infants, is caused by premature fusing of cranial vault sutures, and manifests itself in abnormal skull growth patterns. Left untreated, the condition may lead to severe developmental impairment. Standard practice is to apply corrective cranial bandeau remodeling surgery in the first year of the infant’s life. The most frequent type of surgery involves the removal of the so-called {\em fronto-orbital bar} from the patient's forehead and the cutting of well-placed incisions to reshape the skull in order to obtain the desired result.
  
    In this paper, we propose a precise optimization model for the above cranial bandeau remodeling problem and its variants. We have developed efficient algorithms that solve best incision placement, and show hardness for more general cases in the class. To the best of our knowledge this paper is the first to introduce optimization models for craniofacial surgery applications.
    
    {\small \textbf{Keywords:} OR in medicine, Dynamic programming, Combinatorial optimization, Surgical planning}
\end{abstract}

\section{Introduction}

While adult skulls are a single solid piece of bone, infant skulls are in fact
composed of several bone pieces that start to fuse at the joints ({\em cranial
vault sutures})
as the infant grows. However, when these sutures fuse prematurely, it leads to a condition called “craniosynostosis”.

The purpose of this work
is to study optimization models arising from the treatment of
{craniosynostosis}, which affects 1 in 2000 infants \cite{SL+08}.
Craniosynostosis often yields abnormal growth patterns of the infant’s skull, ultimately resulting in severe deformations (see Figure~\ref{fig:intro}.(a)).
Left untreated, it may also lead to increased intracranial pressure, and this in turn may cause stunted mental growth,
vision impairment and several other severe impairments of the patient. That said, the primary reason for treatment is psychosocial concerns later in life from the abnormal skull shape. To avoid such problems, surgical treatment of craniosynostosis is often applied within the first year of an infant’s life~\cite{PU03}.

\begin{figure}
  \begin{center}
    \includegraphics[width=.95\textwidth]{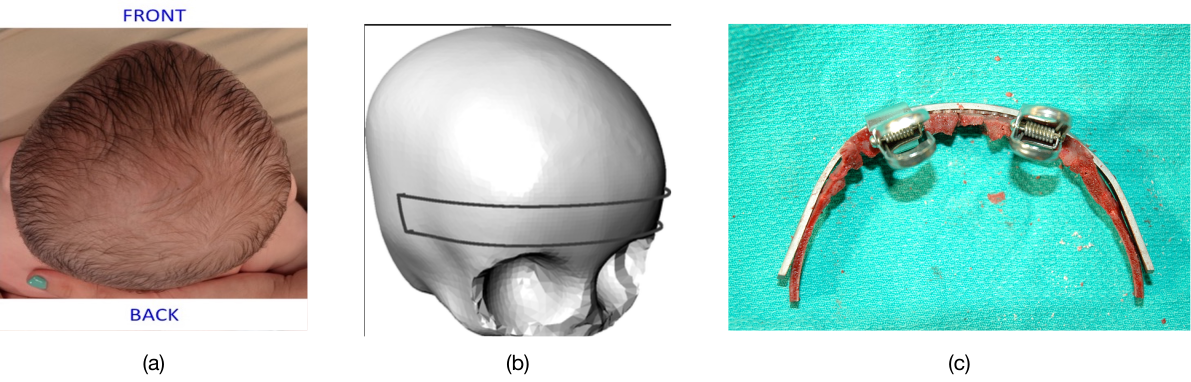}
  \end{center}
  \caption{\label{fig:intro} (a) the skull of a craniosynostosis patient ~\cite{wikicranio},
    (b) the schematics of a bandeau~\cite{SP+12}, and (c) a bandeau attached to metal template~\cite{ik+15}}
\end{figure}

%\begin{wrapfigure}{R}{0.25\textwidth}
%  \includegraphics[scale=0.3]{images/craniosyno_pic.jpg}
%  \caption{Craniosynostosis patient~\cite{wikicranio}}\label{fig:kid}
%\end{wrapfigure}
%\begin{wrapfigure}{R}{0.25\textwidth}
%  \includegraphics[scale=0.35]{images/bandeau.png}
%\caption{Schematics of a bandeau~\cite{SP+12}}\label{fig:bandeau}
%\end{wrapfigure}

To better understand the optimization problem that arises in such surgical treatment, we first focus on one of the methods to treat the condition. A central component in such surgeries is the reshaping of the so-called {\em fronto-orbital bar}, or \emph{bandeau}, which is a strip of bone from the patient’s forehead above the eyebrows (see Figure~\ref{fig:intro}.(b)). The surgeon removes the bandeau from the patient’s skull, then performs several well-placed incisions (`kerfs'~\cite{anantheswar2009pediatric}) into this piece of bone in order to make it more malleable. After kerfing the bone and the bandeau is fit perfectly into the template fixation is achieved using bioresorbable plates and screws.( see Figure~\ref{fig:intro}(c)).
The reshaped bandeau is then inserted back into the patient’s forehead. Note that in the bandeau procedure, the pieces relative position to each other is preserved.
A second reshaping procedure is one over the {\em full vault}, where larger pieces of the whole cranium are cut and may be rearranged and put into completely different positions, so that the original arrangement and  orientation of the pieces need not be preserved.

% \begin{wrapfigure}{R}{0.3\textwidth}
%  \includegraphics[scale=0.2]{images/NRS_JCFS2011_Fig5.png}
%  \caption{Bandeau attached to metal template~\cite{ik+15}}
%  \label{fig:bandeauattach}
%  \end{wrapfigure}

The purpose of these reshaping procedures is to make the bone look closer to what is considered an ideal shape. Thus, one important part of such problems is the determination of the ideal skull shape. Recently, a team of surgeons and researchers at the Hospital for Sick Children developed a library of normative pediatric skulls modelling {\em ideal} skull shapes of infants of varying ages~\cite{SP+12}. Based on these skulls, the hospital now has such steel skull templates and the surgeon chooses the template that is most appropriate to each individual craniosynostosis patient, based on the patient’s age and size. The availability of such individual steel templates during operation has led to a more standardized, objective and precise correction of craniosynostosis~\cite{SP+12}.

A key factor in adequately approximating the ideal skull shape is to strategically place the kerf %\footnote{Surgeons do not cut the bandeau, but rather `kerf'~\cite{anantheswar2009pediatric} it at prescribed cut locations to make it bendable.}
  points on the deformed bandeau. This important part of the surgery is currently performed based on \emph{ad hoc} methods, using the surgeon’s experience, intuition and talent. Obviously, this means that the quality of the outcome of this procedure depends integrally on the surgeon’s experience and skill. Another major drawback is the current difficulty in training new surgeons, and ensuring standardized and consistent surgical outcomes.

In this paper we focus on the surgeon’s task of finding locations for the fronto-orbital bar incisions, the {\em bandeau reshaping problem (BRP)}. We define a formal mathematical problem, the \emph{curve refitting problem} (\ref{prob:cr}), which models this task and allows us to evaluate the quality of an operation's outcome. We present an efficient algorithm for this problem, and show hardness results for a variant of it. In particular, the variant studied is motivated by the full-vault surgery and implies  a hardness result in the full-vault problem. To the best of our knowledge, our work presents the first formal injection of optimization methods into the realm of plastic craniofacial surgery.

\subsection{Related Work}

The use of optimization in healthcare has received increased attention in
recent years. The range of applications has gone from staff and facilities scheduling
(e.g.~\cite{bc+04,cdb10}) and outpatient scheduling~\cite{cayirli2003outpatient} to cancer radiation treatment~\cite{ly+13,frimodig2019models}.
However, to the best of our knowledge, this is the first work that applies an
optimization approach to a
problem that deals with bone
reconstruction.
Indeed, most of the quantitative work in craniosynostosis treatment has focused on measuring the severity of cases~\cite{rs+06,ck+16}.

There have also been some similar problems studied in different contexts. For instance, \cite{ls02} looks at the problem of reconstructing an artifact based on its fragments. However, such work is based on reconstructing a previously known artifact based on how similar the edges of the broken pieces are, which does not apply in our case.
In addition, there are works on computational geometry that study how to best match two sets of points, like the Iterative Closest Point problem~\cite{cv+02} and the Procrustes problem~\cite{gd04}. However these problems only allow some very well-defined uniform transformations (e.g. scaling, rotating) to be applied to the set of points and thus cannot be applied in our setting, where one seeks to maintain some parts of the skull intact, while transforming (rotating) other parts, which are obtained after the incisions are done, independently from the remainder of the skull.

Therefore, none of the previous literature found would apply in solving the desired (BRP)  and thus, there is a need to develop new methodology to solve this problem.

\subsection{Our Contributions and Outline}

We consider the main contribution of our work to be the introduction of formal discrete optimization methods for craniofacial plastic surgery. Moreover, to the best of our knowledge we are the first to provide a formal abstract model (\ref{prob:cr}) of the craniofacial surgical problem of reshaping the fronto-orbital bar, which is presented in Section \ref{sec:compmodels}. We then present an exact algorithm for a practically important special case of the problem based on dynamic programming in Section \ref{sec:dp}, along with the presentation of an efficient implementation of the algorithm and empirical evaluation in Section \ref{sec:comp-exper}. Lastly, we prove in Section \ref{sec:hardness} that the general form of the problem is computationally hard.

\section{Computational Models}\label{sec:compmodels}

\paragraph{}
In this section we present a mathematical model of the Bandeau Reshaping Problem (BRP). The model is very general, accommodating many different cost functions for measuring dissimilarity between the ideal and deformed curves. It also models surgeries in which the pieces of the deformed curve can be rearranged. This is not done in most (BRP) cases in practice, but is studied as a first step in trying to solve the more general full-vault problem. In Subsection~\ref{sec:without-rearrangement} we discuss a specialization of our model which does not admit rearrangement, and is thus more indicative of the bandeau surgery.

\subsection{Curve Reshaping Problem}\label{sec:formulation}
In a Curve Reshaping Problem (\ref{prob:cr}) we are given, as inputs, two piecewise linear continuous functions\footnote{The requirement that our curves are functions is not necessary for our algorithm but greatly simplifies notation.} (henceforth referred to as \emph{curves}) of the form $f:[0,s]\rightarrow \R$ and $g:[0,t]\rightarrow \R$. The curves are used to represent bandeaus from a top-down perspective, with curve $f$ for the deformed bandeau and curve $g$ for the ideal bandeau. When presenting a curve $\eta:[0,u] \rightarrow \R$, we give a set of points $D\subseteq [0,u]$ to the algorithm as input, such that $\eta$ is a piecewise linear interpolation of the list $[(x,\eta(x)): x \in D]$. The set of points $D$ is called the \emph{discretization} of $\eta$.

We are also given as input to (\ref{prob:cr}) an integer $k\geq 1$, indicating the number of cuts on $f$ we intend to make. The last piece of input we are given is a function $c:2^{[0,s]}\times [0,t]\times[0,t]\rightarrow \R_+\cup\{\infty\}$. For any $a<b \in [0,s]$, $l<r\in [0,t]$, the value of $c([a,b],\ell, r)$ measures the cost of using the segment $\restr{f}{[a,b]}$ of $f$, in order to cover a segment of $g$ by clamping the left endpoint of $\restr{f}{[a,b]}$ at $\ell$ and the right endpoint at $r$. See Subsection \ref{sec:diss} for an elaboration, where we discuss a particular cost function used in our target application. In our models we treat the function $c$ as a black box oracle.

In a (\ref{prob:cr}) problem, the decisions are as follows: Choose $k$ positions on the deformed curve, $p_1 <\dots <p_{k} \in P,$
where $P \subseteq [0,s]\cap \Q$ is the discretization of $f$. For simplicity we assume that $p_1 > 0$ and $p_k <s$ and let $p_0=0, p_{k+1}=s$. These cuts segment $f$ into pieces $\restr{f}{[p_0,p_1]}, \dots, \restr{f}{[p_k,p_{k+1}]}$, which will be mapped onto the ideal curve $g$. We index the resulting pieces of the deformed curve by the set $\{0\} \cup [k]$. Here $[k]$ is a common shorthand notation for $\{1,\dots, k\}$.

Let $Q\subseteq [0,t] \cap \Q$ be the discretization of $g$. We map the segments of $f$ onto $g$ by making the following decisions: for each $i \in \{0\} \cup[k]$, we select two positions $\ell_i, r_i \in Q$ indicating positions of the left and right endpoints that the $i^{th}$ segment of $f$ will take on the ideal curve.

The objective in our (\ref{prob:cr}) problem is to minimize the total dissimilarity induced by our solution. We introduce the notion of coverage. Suppose $$Q=\{ q_{1}, \dots, q_{|Q|} \}.$$ For $j \in [|Q|-1]$ we say that the interval or segment $[q_{j}, q_{j+1}]$ on $g$ is \emph{covered} by our solution if there is some $i \in J$ such that $$\min(\ell_{i},r_i) \leq q_{j} \leq q_{j+1} \leq \max(\ell_i,r_{i}).$$ In this case piece $i$ \emph{covers} $[q_{j}, q_{j+1}]$. Let $\uncov(\ell, r)$ be the set of segments $[q_{j}, q_{j+1}]$ on $g$ that are not covered by our solution. Note that since our clamp points $\ell_i, r_i$ for piece $i$ are elements of $Q$ our notion of coverage includes any intervals of the ideal curve located between where we clamped the endpoints of our curve segment.

To penalize not covering a segment of the ideal curve we select a parameter $\delta \geq 0$ and add the term $$\delta |\uncov(\ell,r)|$$
%\sum_{j=1}^{|Q|-1} \max(0, 1 - \sum_{i \in J} |\{ \{ q_{j}, q_{j+1} \} \} \cap cov(\ell_{i}, r_{i})|)
In summary, we may formally state our problem as follows:
 \begin{align*}\tag{CR}\label{prob:cr}
 \min& \sum_{i \in J} c([p_{i},p_{i+1}],\ell_i, r_i) + \delta |\uncov(\ell,r)|\\
 \text{s.t. }
 & p_1<\dots<p_k \in P \\
 &\ell, r \in Q^{\{0\} \cup [k]}.
 \end{align*}

 Notice that (\ref{prob:cr}) requires exactly $k$ cuts to be made on the deformed curve. If one wanted find the optimal solution with at most $k$ cuts they can simply solve (\ref{prob:cr}) for different cut numbers from $0$ to $k$.

% One can generalize the (\ref{prob:cr}) problem assumptions by replacing our assumptions about uncoverage and overlap with general penalty functions for: leaving a piece of the ideal curve uncovered, leaving a piece of the deformed curve unused, and having overlap in how your segments of the deformed curve cover the ideal curve. We leave these generalizations to future work.

%In either case, we assume an algorithm for our (\ref{prob:cr}) has access to an oracle that can evaluate the value of $f(x)$ (or $g$ respectively) for any $x\in P$ ($x \in Q$ respectively) in polynomial time.

 \subsection{Without Rearrangement}\label{sec:without-rearrangement}
\paragraph{} In a general $3$-dimensional cranial vault remodeling surgery, where the entire skull is being reorganized, it is possible to take pieces and rearrange their relative positions. The (\ref{prob:cr}) model as described previously also admits this possibility. In practice, the (BRP) surgery for reshaping the bandeau does not admit this. That is to say, the relative positions of the cut segments of the deformed curve are not changed by their mapping onto the ideal curve. In fact the surgeons do not break apart the bandeau along the points where they cut it, rather they bend the bandeau into its proper shape at these points. To model this we define the (\ref{prob:cr}) problem \emph{without rearrangement}. In this problem we add the constraints: 
\begin{itemize}
  \item $\ell \leq r$ to indicate that the right endpoint of segment $i$ needs to be placed further right than the left endpoint of segment $i$, i.e.\ that segments cannot be flipped.
  \item for all $i \in [k]$, $r_{i-1} = \ell_{i}$ to indicate that the $i$-th segment is mapped to the right of the $i-1$-th segment.
\end{itemize}

In Section~\ref{sec:dp} we will give a polynomial time algorithm based on dynamic programming for the (\ref{prob:cr}) problem without rearrangement. This will be contrasted with a proof in Section~\ref{sec:hardness} that the (\ref{prob:cr}) is strongly NP-hard in general.

\section{Dynamic Programming Algorithm}\label{sec:dp}
\paragraph{}
In this section, we present a dynamic programming algorithm for (\ref{prob:cr}) problems without rearrangement. For the remainder of this section, we consider an instance $I$ of the (\ref{prob:cr}) problem without rearrangement, where the deformed curve $f:[0,s]\rightarrow \R$ and the ideal curve $g:[0,t]\rightarrow \R$ are discretized as point sets $P$ and $Q$ respectively.
Let $c$ be the associated cost function, $\delta \geq 0$ the uncoverage parameter, and let $k\geq 1$ be an integer indicating the associated number of cuts.

We intend to break instance $I$ of (\ref{prob:cr}) into a set of suitable subproblems so that we can impose a dynamic programming paradigm.
Choose an integer number of cuts $k' \in [0,k]$.
Choose $a,b \in [0,s] \cap P$ such that $a \leq b$ and $\size{[a,b] \cap P} \geq k' + 2$.
This is a subinterval on the deformed curve on which we can make $k'$ cuts excluding the endpoints.
Similarly choose $d,e \in [0,t] \cap Q$ such that $d\leq e$ and $\size{[d,e] \cap Q} \geq k'+2$.
Now let $\text{CR}([a,b]],[d,e], k')$ denote the optimal value of (\ref{prob:cr}) problem without rearrangement using cost function $c$, on the deformed curve $\restr{f}{[a,b]}$ with feasible cut positions $P\cap [a,b]$ and ideal curve $\restr{g}{[d,e]}$ with feasible clamp positions $Q \cap [d,e]$ and following added constraints:
\begin{enumerate}
  \item we restrict to feasible solutions which use exactly $k'$ cuts, and
  \item we restrict to feasible solutions for which $\uncov(\ell, r)$ is empty.
\end{enumerate}
By our choice of parameters, which we call a \emph{feasible} choice of parameters, $\text{CR}([a,b],[d,e],k')$ always has a feasible solution.
Our goal is to give a recursive formula for $\text{CR}([a,b],[d,e],k')$ that can be used to design a dynamic program.
\begin{theorem}\label{th:dp-opt}
For any feasible choice of parameters the following recursive equation holds
$$\text{CR}([a,b],[d,e],k')=$$
$$=\begin{cases}
  {\displaystyle\min_{(p,q) \in P\cap (a,b)\times Q\cap (d,e)}} \left(c([a,p], d,q) + \text{CR}([p,b], [q,e], k'-1)\right), &
  \text{ if $k' > 0$ }\\
  {\displaystyle c([a,b], d,e)}& \text{ otherwise.}
  \end{cases}
  $$
\end{theorem}

\begin{proof} 
If we expand the recursive expression on the right hand side of the equation, we obtain the cost of a feasible solution to the problem for which $\text{CR}([a,b], [d,e]), k')$ is the optimal value, and hence the right hand side expression is an upper bound on the optimal value $\text{CR}([a,b], [d,e], k')$. It remains to show it is also a lower bound.
  
We proceed by induction on $k$. If $k'=0$ then there is no decision to make. The only feasible solution is to map $\restr{f}{[a,b]}$ onto $\restr{g}{[d,e]}$ and pay cost $c([a,b],d,e)$ as desired.

Hence we may assume $k' > 0$ and the inductive hypothesis holds on all $\text{CR}(\cdot,\cdot,k'')$ where $k'' < k'$.  Let $p_1^*, \dots, p_{k'}^* \in P\cap(a,b)$ and $(\ell_1^*,r_1^*),\dots, (\ell^*_{k'},r_{k'}^*) \in Q^2\cap (d,e)^2$ be an optimal solution to $\text{CR}([a,b],[d,e],k')$. Observe that
   \begin{align*}
     \text{CR}([a,b],[d,e],k') &= \sum_{i=1}^{k'} c([p^*_{i-1},p^*_i], \ell^*_i,r^*_i) = c([a, p^*_1],d, r^*_1) + \underbrace{\sum_{i=2}^{k'} c([p^*_{i-1},p^*_i],\ell^*_i,r^*_i)}_{\geq CR([p_1^*,b], [r^*_1,e],k'-1)} \\
     &\geq \min_{(p,q) \in P\cap(a,b)\times Q\cap(d,e) } \left(c([a,p],d,q) + \text{CR}([p,b], [q,e], k'-1)\right).
   \end{align*}
   The underbrace inequality above follows since $p_2^*,\dots, p_{k}^*$ and $(\ell_2^*,r_2^*),\dots, (\ell^*_{k'},r_{k'}^*)$  is a feasible solution to the (\ref{prob:cr}) problem without rearrangement using exactly $k'-1$ cuts on curves $\restr{f}{[p_1^*,b]}$ and $\restr{g}{[\ell_2^*,e]}$ which pays no uncoverage penalty. Note that since this problem is without rearrangement, $\ell_1^* < r^*_1$ and $r^*_1 = \ell^*_2$.
\end{proof}
\paragraph{}
 By Theorem \ref{th:dp-opt}, when $\delta = 0$, $I$ can be solved using standard dynamic programming techniques utilizing a table of size
$O(|P|^2|Q|^2k)$. Each table entry can be computed in $O(|P||Q|)$
time, yielding an algorithm solving (\ref{prob:cr}) without rearrangement in time $O(|P|^3|Q|^3k)$. Since $k$ is at most $\min\{|P|,|Q|\}$ this is a polynomial time algorithm. In fact, the experiments in Section \ref{sec:comp-exper} demonstrate that a very small number of cuts are needed relative to the size of $P$ and $Q$ in order to obtain an extremely high quality bandeau reconstruction.

For $\delta >0$, observe that the uncoverage penalty $\gamma(p_0 +(s-p_k))$ is completely determined once $\ell_0$ and $r_k$ are fixed because we are working in the paradigm which disallowed rearrangement. Since $r_{i-1}= l_i$, and $l <= r$, the only possible uncovered parts of $g$ are the ones on points less than $l_0$ or greater than $r_{k}$, which gives precisely the uncoverage penalty above.  Then it remains to solve a (\ref{prob:cr}) problem with optimal value $\text{CR}([0,s], [\ell_0, r_k], k)$. By enumerating our choices for $\ell_0,r_k$ and solving the subproblems using Theorem \ref{th:dp-opt}, we get the following corollary:
\begin{corollary}
Any (\ref{prob:cr}) problem without rearrangement can be solved in time polynomial in $|P|,|Q|$ and linear in $k$.
\end{corollary}

\section{Computational Experiments}\label{sec:comp-exper}

In this section, we describe results obtained by testing our dynamic programming algorithm for (\ref{prob:cr}) without rearrangement on a test suite of craniosynostosis cases. Due to medical privacy regulations and institutional policy it is difficult and time-consuming to obtain scans of actual patient skulls from operations performed at the Hospital for Sick Children. We were able to obtain three such test cases including both pre-op and post-op scans with all patients' personal identifying information removed. Results from the tests performed with those are described in Subsection~\ref{sec:prepost}. To augment these results we designed a mathematical model to generate a large number of synthetic test cases. In Subsection~\ref{sec:synth-model} we describe the model and test cases resulting from it. In Subsection~\ref{sec:synth-results} we report on the results of applying our algorithm to the test cases generated by our synthetic model. Subsection~\ref{sec:diss} discusses the particular choice of cost function for problem (\ref{prob:cr}) we implemented in our tests.

For our tests we ran an implementation of algorithm discussed in Section~\ref{sec:dp} written in Python 3.6.7 using Numpy 1.15. The tests were run on a Macbook Pro with a 2.6GHz 6-Core Intel i7 processor and 16GB of DDR4 RAM, running macOS Catalina version 10.15.7.
\subsection{Computing the Dissimilarity}\label{sec:diss}
\begin{wrapfigure}{r}{0.5\textwidth}
  \centering
  \subfloat[Curve $f$]{
  \begin{tikzpicture}[scale=0.23]\draw[very thick,rotate=30](0, 0) to [out = 50, in = 185] (3, 3) to [out = 5, in = 120] (6, 2.6) to [out=300, in = 120] (8.4, 1.3) to [out = 279, in = 160] (10, 0);
  \draw [very thick, dashed, ->,rotate=30] (0,0) to (10,0);
  %\draw[very thick, dashed, ->] (6, 2.5) to (5.65,0.8);
  \node [below] at (0,0) {$L_f$};
  \node [below] at (8.5,4.5) {$R_f$};
  \end{tikzpicture}
  }
  \hspace{\fill}
  \subfloat[Curve $g$]{
  \begin{tikzpicture}[scale=0.23]
  %\draw [help lines] (0,0) grid (10,6);
  \draw [very thick] (0, 3.5) to[out = 50, in =180] (6,6.5) to [out = 0, in = 110] (10, 3.5);
  \draw [very thick, dashed, ->](0,3.5) to (10, 3.5);
  %\draw [very thick, dashed, ->](1.5, 5) to (5,6.4);
  %\node [below] at (5,5) {Acceptable};
  \node [below] at (0,3.5) {$L_g$};
  \node [below] at (10,3.5) {$R_g$};
  \end{tikzpicture}}
  \hspace{\fill}
  \subfloat[$d(f,g)$]{
  \begin{tikzpicture}[scale=0.23]
  \draw [very thick] (0, 3.5) to[out = 50, in =180] (6,6.5) to [out = 0, in = 110] (10, 3.5);
  \draw[very thick,fill=gray](0, 3.5) to [out = 50, in = 185] (3, 6.5) to [out = 5, in = 120] (6, 6.1) to [out=300, in = 120] (8.4, 4.8) to [out = 279, in = 160] (10, 3.5)
  to [out = 110, in =0] (6,6.5) to [out=180, in = 50 ] (0,3.5)
  -- cycle
  ;
  %\draw [very thick] (0, 3.5) to[out = 50, in =180] (6,6.5) to [out = 0, in = 110] (10, 3.5);
  %\draw[very thick](0, 0) to [out = 50, in = 185] (3, 3) to [out = 5, in = 120] (6, 2.6) to [out=300, in = 120] (8.4, 1.3) to [out = 279, in = 160] (10, 0);
  %\node[below] at (4, 2) {Not};
  \end{tikzpicture}
  }
  \caption{Calculating $d(f,g)$}
  \label{fig:dfg}
  \end{wrapfigure}
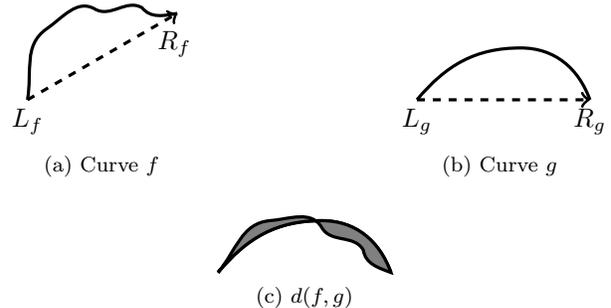
%We assume that the malformed and the ideal curves are $x$-monotone piecewise linear curves in $\mathbb{R}^2$. Such a curve $h$ can be thought of as being a function $h : [x(L_h), x(R_h)] \rightarrow \mathbb{R}$, where $L_h$ and $R_h$ are its left and right endpoints respectively. Note that $h$ can be specified by a list of inflection points $r_1, \dots, r_{\alpha_h} \in \mathbb{R}^2$, where $r_1 = L_h$, $r_{\alpha_{h}} = R_h$, and $x(r_1) < x(r_2) < \dots < x(r_{\alpha_h})$.

%Given two points $p, q$ on a curve $h$, where $x(p) \le x(q)$, we denote by $\restr{h}{p, q}$ the subcurve of $h$ comprised between $p$ and $q$.

To evaluate the quality of a solution, we define, for any pair of curves $f:[0,s]\rightarrow \R$ and $g:[0,t]\rightarrow \R$, a dissimilarity measure $d(f, g) \in \mathbb{R} \cup \{\infty\}$. Note that the dissimilarity measure is not necessarily symmetric, that is, the measures $d(f, g)$ and $d(g, f)$ are not necessarily equal. Before proceeding we reiterate that the algorithm does not rely on any properties of $d$, and hence can be applied to different measures at the discretion of the user.

We will use $L_f$ and $R_f$ as shorthands for the left and right
endpoints $(0,f(0))$, and $(s,f(s))$ of $f$, respectively. We say that two curves $f$ and $g$ \emph{match} if the Euclidean
distance between the endpoints of $f$ is approximately equal to the
Euclidean distance between the endpoints of $g$. Formally, we
introduce a parameter $\alpha \in [0, 1]$, and we say that $f$ and $g$
match if
$
1 - \alpha \le \frac{\left \lVert L_g - R_g \right \rVert_2}{\left \lVert L_f - R_f \right \rVert_2} \le 1 + \alpha \ .
$

If $f$ and $g$ do not match, then we define $d(f, g) = \infty$. Next, we describe how to compute $d(f, g)$ when they do match. First, we rotate $g$ around $L_g$ so that its right endpoint is moved to the point $
\left(x(L_g) + \left \lVert L_g - R_g \right \rVert_2, y(L_g) \right) $, and we denote by $\tilde{g}$ the curve thus obtained. Next, we modify the curve $f$ as follows. We scale it by a factor of $\frac{\left \lVert L_{\tilde{g}} - R_{\tilde{g}} \right \rVert_2}{\left \lVert L_f - R_f \right \rVert_2}$, so that the Euclidean distance between its endpoint becomes equal to the Euclidean distance between the endpoints of $\tilde{g}$. Next, we translate and rotate it so that its left and right endpoints are mapped to $L_{\tilde{g}}$ and $R_{\tilde{g}}$ respectively. We denote by $\tilde{f}$ the curve thus obtained. Finally, we define $d(f, g)$ to be the area between the modified curves $\tilde{f}$ and $\tilde{g}$. If $\tilde{f}$ and $\tilde{g}$ are functions we can write
$
d(f, g) = \int_{x(L_{\tilde{f}})}^{x(R_{\tilde{f}})} | \tilde{f}(x) - \tilde{g}(x) | \ \mathrm{d}x \ .
$

Unfortunately we cannot guarantee that $\tilde{f}$ and $\tilde{g}$ are functions, so we draw on methods in computational geometry. Since $\tilde{f}$ and $\tilde{g}$ are piecewise linear and their endpoints line up, their union forms the boundary of a closed polygon $A$. By computing where the edges of $A$ intersect we can partition $A$ into a set of simple polygons in time bounded by a polynomial in the number of points in the discretization. Using the well-known Shoelace Method~\cite{meister1769generalia} we can compute the area of each such simple polygon and sum them to obtain the area of $A$. This area we denote $d(f,g)$.
Figure~\ref{fig:dfg} shows an example of two curves $f$ and $g$ and how $d(f,g)$ is calculated.

Now we are ready to define a cost function $c:2^{[0,s]}\times [0,t]\times[0,t]\rightarrow \R\cup\{\infty\}$ for the (\ref{prob:cr}) problem. For all $[a,b]\subseteq [0,s]$ and all $[\ell,r] \subseteq [0,t]$, we define
$$c([a,b],\ell,r) =
  d(f|_{[a,b]}, g|_{[\ell,r]}).
$$

Note that curves presented via an explicit discretization, such as those presented to the DP algorithm implied by Theorem \ref{th:dp-opt} are piecewise linear functions whose inflection points are all part of the input. Thus evaluating $d$ on pairs of such curves can be done in polynomial time (in the number of pieces), and hence $c$ can be queried in polynomial time. In our runs of the algorithm, we pre-compute a table of costs $c$ and assume $O(1)$ access to its elements during operation of the dynamic program.

\subsection{Comparing Our Algorithm to Surgical Outcomes}\label{sec:prepost}

The Hospital for Sick Children provided us with 3d images of three pre-operation (pre-op) and three corresponding post cranial vault remodelling operation (post-op) skulls of metopic\footnote{Metopic is a medical term referring to a common variant of craniosynostosis.} cases where they performed surgical interventions. The post-op scans are taken between $3$-$5$ days after the surgery. From these scans we created deformed curves plotted from a top-down perspective as described in Section~\ref{sec:compmodels}. Figure~\ref{fig:prepost-top-down} shows the resulting curves from one such case.

\begin{figure}
  \centering
  \includegraphics[scale=0.45]{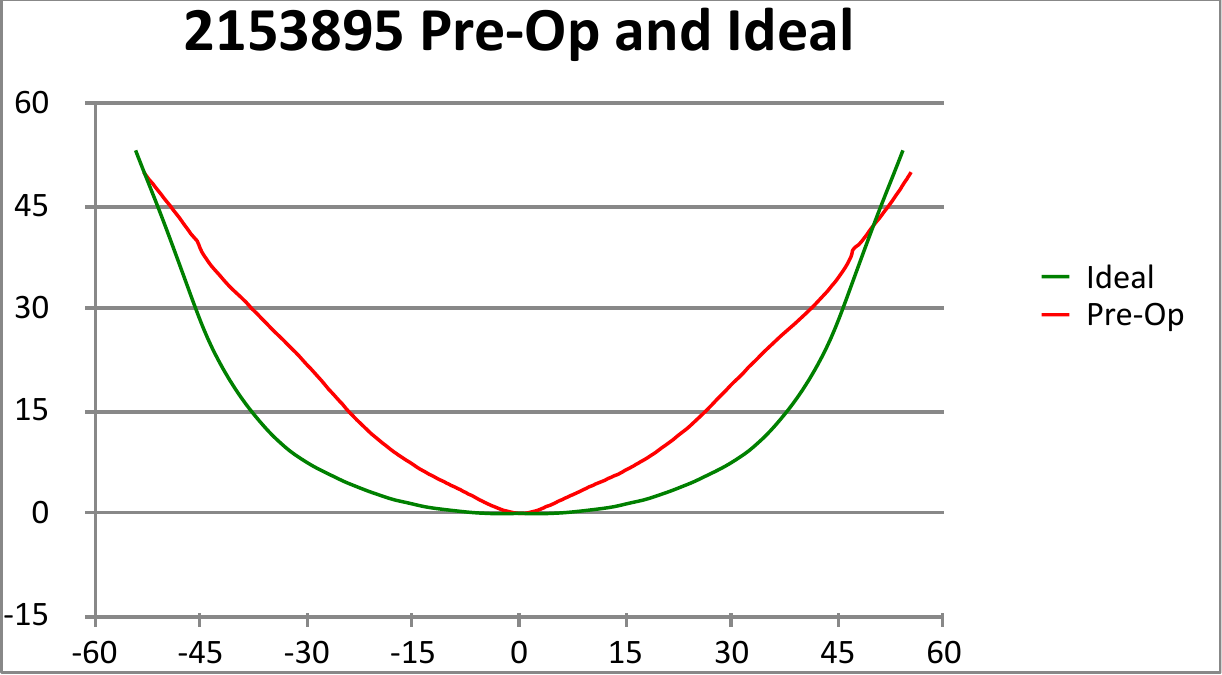}
  \includegraphics[scale=0.45]{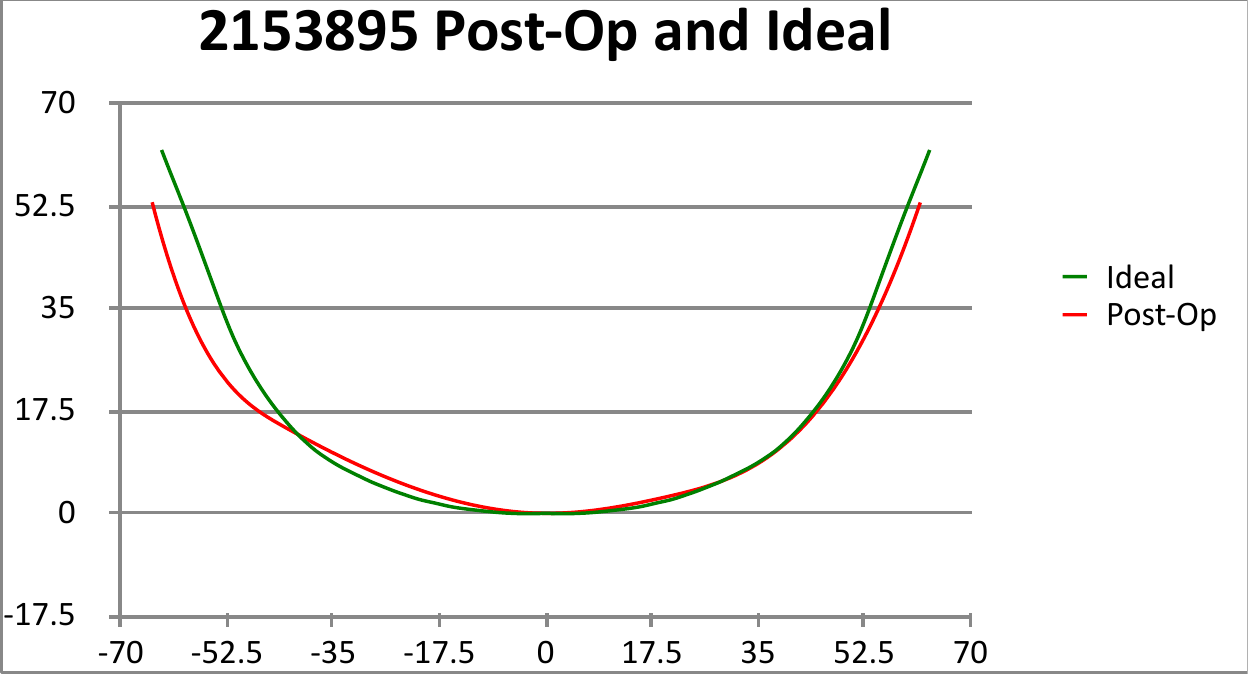}
  \caption{A top-down plot comparing Pre-Op and Post-Op Curves to Ideal}\label{fig:prepost-top-down}
\end{figure}

The curves are given to the algorithm as an explicit list of points describing a piecewise linear curve. The discretization was chosen so that the distance between two consecutive points is between $0.5$ and $1$mm. This resulted instances where the deformed curves had $200$ coordinates. On these three instances our algorithm took on average $0.323$ seconds to find the optimal solution for up to $13$ cuts. Figure~\ref{fig:pre_post} shows a plot of the optimal value of the objective function (Area Below Curve or ABC, see Section~\ref{sec:diss}) as a function of the number of cuts our algorithm applies for the cases provided by Sick Kids. The Post-Op ABC's are provided to show when our algorithm crosses the ABC achieved by the surgeons. In each case, only three cuts are needed for the algorithm to achieve a lower ABC than that achieved by the surgeon.
Although the surgeons did not record the specific number of cuts made in an actual operation, looking at the data provided, the surgeons estimate that at least $7$ cuts were made to achieve the Post-Op curves.

In Figure~\ref{fig:prepost_3} we see our algorithm's surgical plan using three cuts for a case provided by Sick Kids. The green curve is a top-down view of the ideal curve, and the red is a top-down view of the deformed curve. The blue curve is the curve which results from performing the surgery as recommended by the algorithm. The red squares denote the cut positions on the deformed curve, and the blue circles denote the corresponding positions on the ideal curve these cuts are mapped to.

\begin{figure}
  \centering
  \includegraphics[scale=0.8]{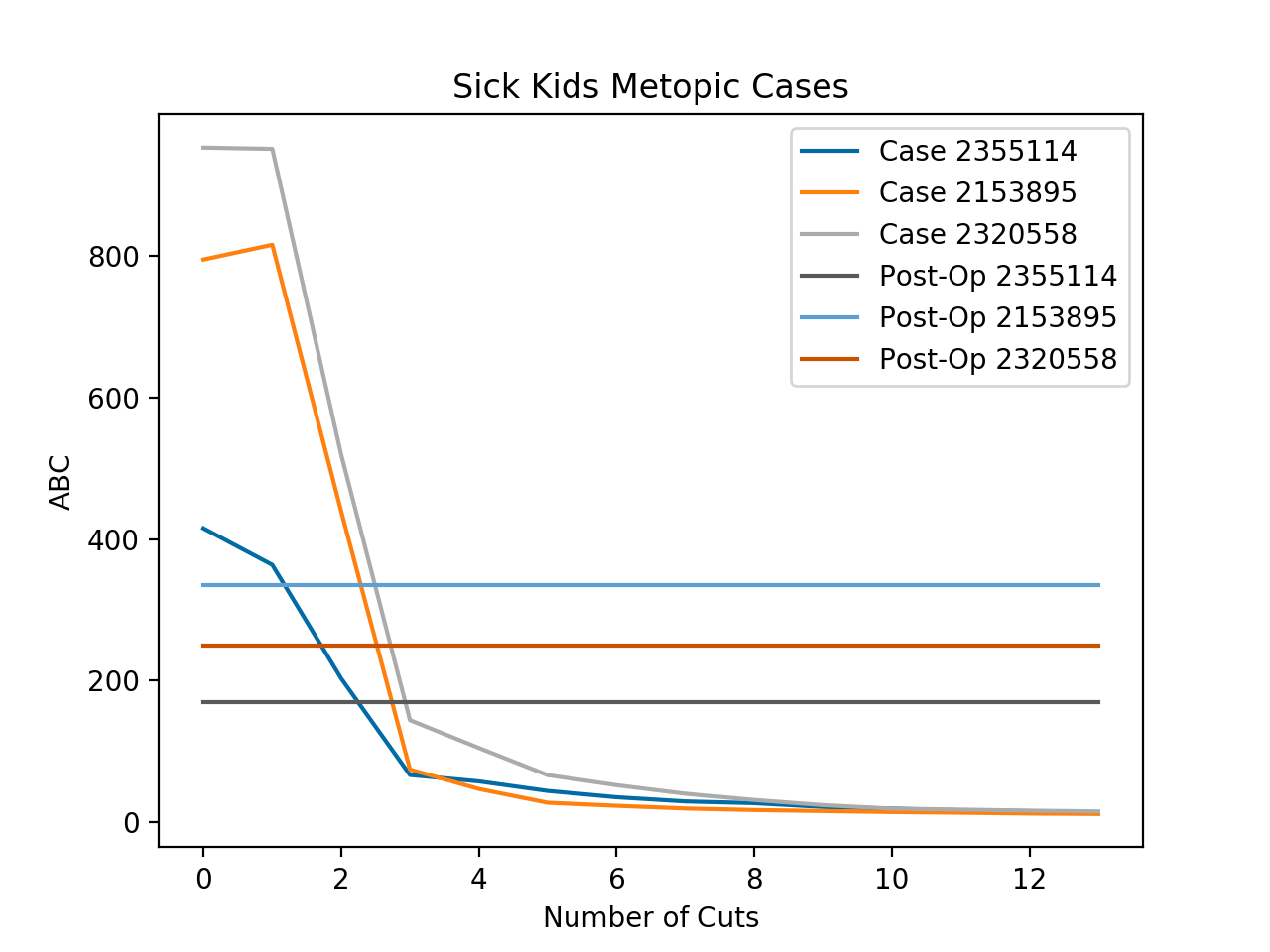}
  \caption{Cuts vs Area Below Curve (ABC) for Hospital For Sick Children Cases}\label{fig:pre_post}
\end{figure}

\begin{figure}
  \centering
  \includegraphics[scale=0.8]{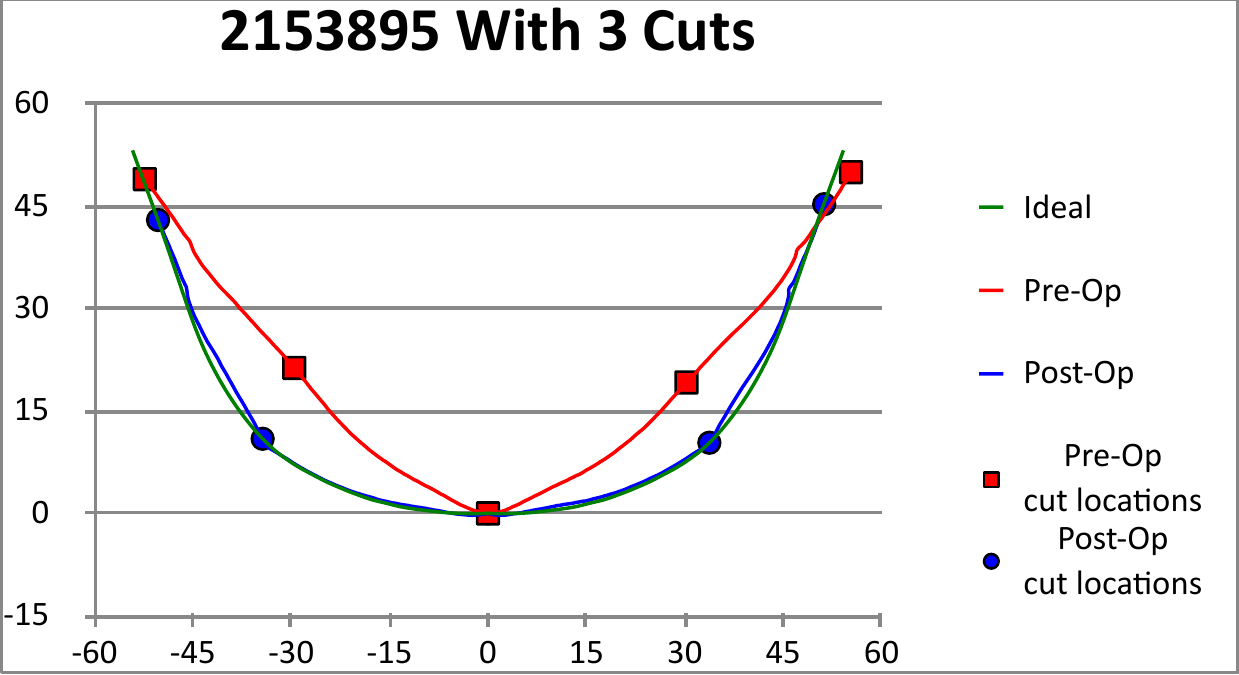}
  \caption{Algorithm's Surgical Recommendation for Case 2153895 With $3$ Cuts}\label{fig:prepost_3}
\end{figure}

\subsection{Synthetic Craniosynostosis Model}\label{sec:synth-model}

In the interest of collecting a larger data set on which to test our algorithm we created a synthetic mathematical model of craniosynostosis cases which can be used to generate instances of (\ref{prob:cr}). Our  model is based on two observations. The first is that craniosynostosis cases are often characterized by a few distinct regions where the curve exhibits different properties. For example, in metopic craniosynostosis the premature fusion of the suture in the middle of the forehead leads to steep curve segments coming to a point near the origin, and more shallow curve segments as the points get further from the origin. The second observation is that surgeons at the Hospital for Sick Children choose the scale of the ideal bandeau so that the furthest points from the origin, as well as the origin, intersect the deformed curve at the furthest points from the origin and the origin respectively. In most practical cases this is always possible.

Formally, for deformed point set $P$, the ideal point set $Q$ is chosen from a collection of scaled ideal curves so that
$$ \{\arg\min\{x: (x,y) \in P\}, (0,0), \arg\max\{x: (x,y)\in P\}\}\subseteq Q.$$
If we fix the scaling of the ideal curve then this means we have a specific set of identified coordinates we need deformed curves to intersect, and in the regions between the intersecting points, the specific type of craniosynostosis case will dictate the steepest or degree of the polynomial representing the curve.

The ideal bandeau provided by Sick Kids is a smooth $2$-dimensional curve intersecting the origin, with endpoints $(-x,y)$ and $(x,y)$ for some $x,y \in \R_{>0}$. The choice of coordinates $(x,y)$ may vary as the bandeau can be scaled to best fit each individual patient. The baseline bandeau without scaling we were given had $x = y = 50.0$.

The \emph{Synthetic Craniosynostosis Model} is defined as follows. Let $(x_0,y_0)$ denote the positive endpoint of ideal curve $g$ intersecting the origin. A synthetic deformed curve $f$ can be defined by specifying two things. The first is a list of points in $\R^2$, $L = \{\ell^1, \dots, \ell^{|L|}\}$, which is monotonically increasing in the first coordinate and contains $(-x_0,y_0)$, $(0,0)$, and $(x_0,y_0)$. The second is a list of rational numbers, $D = \{d_1, \dots, d_{|L|-1}\}$, of size $|L|-1$. For any $x \in [-x_0, x_0]$, the value of $f(x)$ is defined to be
$$f(x): = 
  a_i x^{d_i} + b_i, \quad\text{where $i \in \{1,\dots, |L|-1\}$ such that $x \in [\ell^i_0, \ell^{i+1}_0]$}$$
where each $a_i$ and $b_i$ are chosen to satisfy the equations
\begin{align*}
a_i(\ell_0^i)^{d_i} + b_i &= \ell_1^i \\
a_i(\ell_0^{i+1})^{d_i} + b_i &= \ell_1^{i+1}.
\end{align*}

Intuitively, $f$ is a piecewise curve consisting of $|L|$ segments, where the $i$-th segment is chosen as a curve of the form $ax^{d_i} + b$ such that it intersects the $i$-th and $i+1$-th points in the list $L$. The example in Figure~\ref{fig:synthetic-example} shows a synthetic deformed curve with $$L = \{(-49.3, 48.7), (-12.5, 7), (0,0), (12.5,7), (49.3, 48.7)\}$$ and $D = \{2,1,1,2\}.$

The attentive reader will notice that $f$ defined this way is not piecewise linear. To make the curve given as input to a (\ref{prob:cr}) problem piecewise linear, we discretize it by sampling a set of points $D \subseteq [\ell^1, \ell^{|L|}]\times \text{range}(f)$ uniformly spaced out in their first coordinate and approximating $f$ by the piecewise lienar interpolation of $D$. Of course, the greater the size of $D$ the better the approximation of $f$. In our examples we chose $|D|= 200$ since the ideal curves we were working with also had a discretization using $200$ points.

\begin{figure}
  \centering
  \includegraphics[scale=0.8]{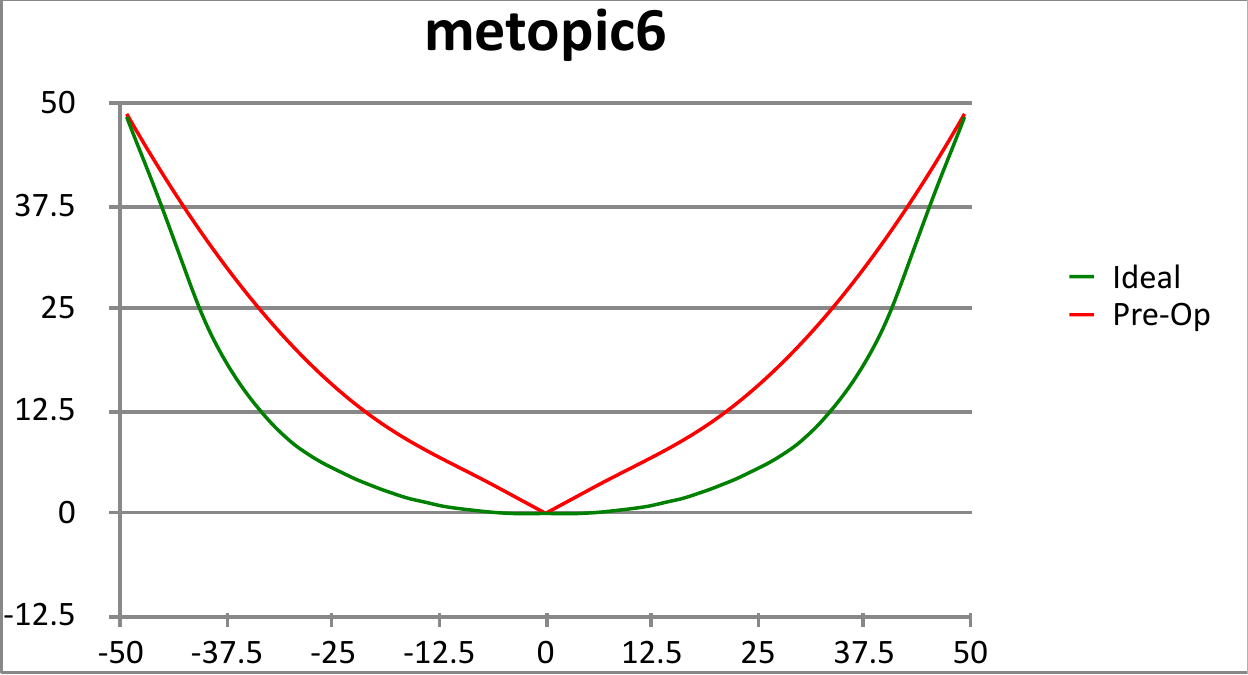}
  \caption{A Synthetic Metopic Case with a split-points at $(0,0)$ and $(\pm 12.5, 7)$.}\label{fig:synthetic-example}
\end{figure}

\subsection{Synthetic Model Test Results}\label{sec:synth-results}

Using the Synthetic Craniosynostosis Model discussed in Subsection~\ref{sec:synth-model} we created three buckets of test cases to reflect different types of craniosynostosis.
The data can be inspected at \href{https://github.com/wjtoth/optimized-cranial-bandeau-remodeling}{https://github.com/wjtoth/optimized-cranial-bandeau-remodeling}.

The first bucket consisted of $24$ synthetic metopic cases. To create these cases, for each test case we sampled an $(x,y)$ coordinate in $[12.5,25]\times [2,20]$ and chose $L$ to be
$$ \{(-49.3, 48.7), (-x, y), (0,0), (x,y), (49.3, 48.7)\}$$
and we sampled two integers $d_1 \in \{1,2\}$ and $d_2 \in \{2,3\}$ and chose $D$ to be $\{d_2,d_1,d_1,d_2\}$. This creates curves where the  split-point $(x,y)$ lies above the ideal curve, and the segments intersecting the origin are steeper than the segments which do not. This simulates the premature fusing of suture in the middle of the forehead which characterizes metopic craniosynostosis. An example can be seen in Figure~\ref{fig:synthetic-example}.

The second bucket consisted of $24$ synthetic sagittal cases. These cases are characterized by a premature fusing of the sagittal sutures~\cite{wikisagittal}, resulting in a skull shape where the forehead appears more flat and wide from a top down perspective. To simulate this phenomenon we created each test case by sampling an $(x,y)$ coordinate in $[27,55]\times [0.5,2]$ and two integers $d_1 \in \{1,2\}$ and $d_2 \in \{2,4\}$. We constructed $L$ and $D$ from these samples in the same manner as we did for the metopic cases above. The result is test cases where the  split-point lies below the ideal curve, and the curve is mostly flat until it reaches the  split-point, then bends up quickly to intersect the ideal endpoints. An example can be seen in Figure~\ref{fig:synthetic-example-sag}.

\begin{figure}
  \centering
  \includegraphics[scale=0.8]{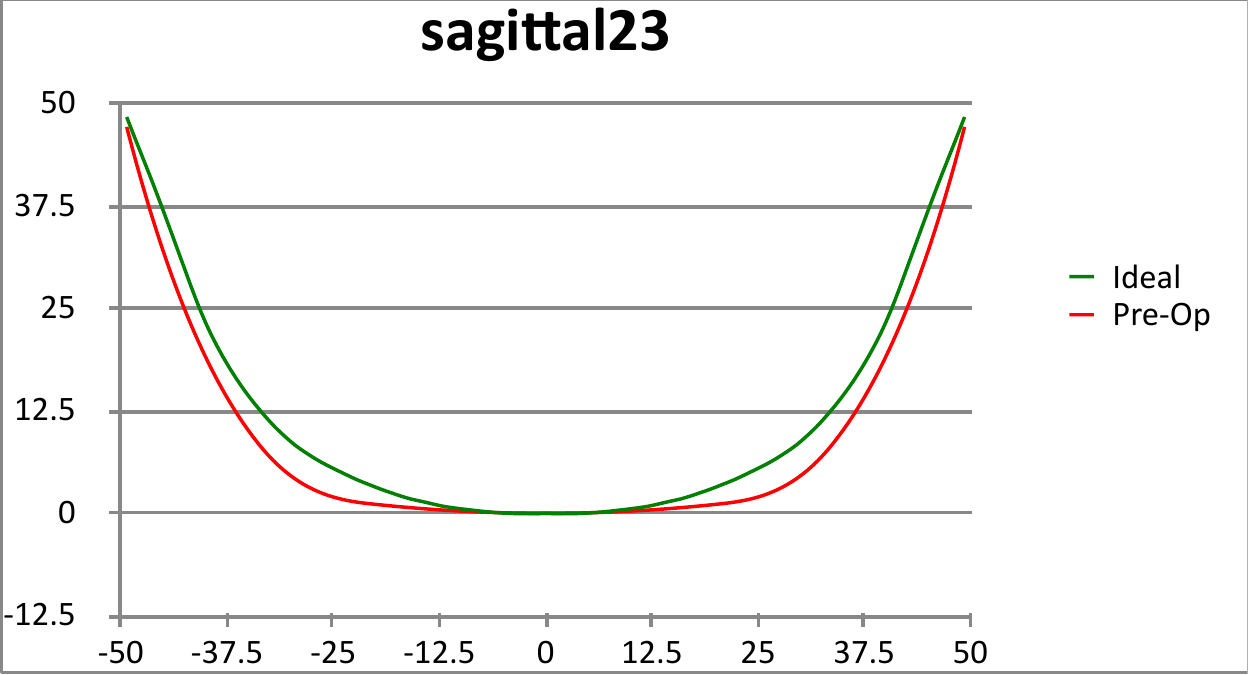}
  \caption{A Synthetic Sagittal Case}\label{fig:synthetic-example-sag}
\end{figure}

The last bucket consisted of $20$ extreme cases. The purpose of these cases was to test the robustness of our algorithm. We emphasize that these cases are not reflective of what real life cases may look like. To generate each case we sampled $5$ random  split-points $(x_1,y_1) \dots (x_5, y_5) \in [1,49]^2$ sorted in order of increasing $x$ value, and sampled $6$ random degrees $d_1,\dots, d_6 \in \{1,2,3,4\}$. Then we chose $L$ to be
$$\{(-50,50), (-x_5,y_5),\dots, (-x_1, y_1), (0,0), (x_1,y_1), \dots, (x_5,y_5), (50,50\}$$ and chose $D$ to be $\{d_6,\dots, d_1, d_1, \dots, d_6\}$.  An example can be seen in Figure~\ref{fig:synthetic-example-extreme}.

\begin{figure}
  \centering
  \includegraphics[scale=0.8]{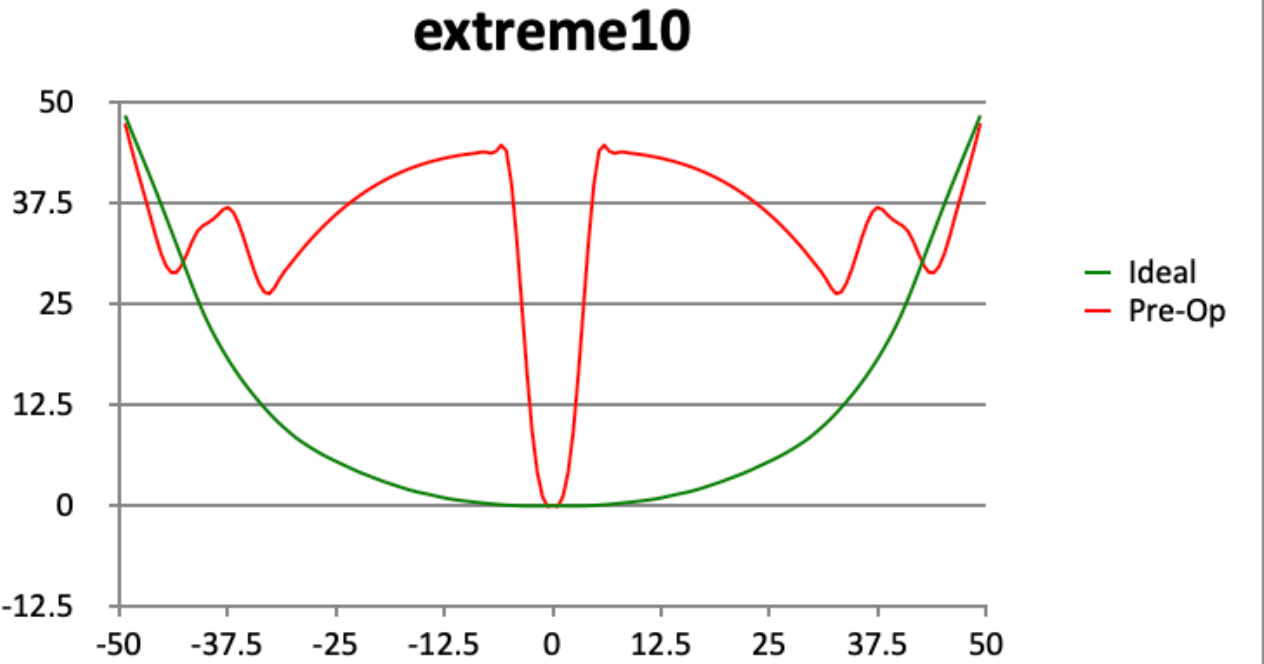}
  \caption{A Synthetic Extreme Case}\label{fig:synthetic-example-extreme}
\end{figure}

After discretization all of our synthetic test cases had on the order of $200$ points. The average computation time for the dynamic program in the metopic bucket was $0.13$ seconds, in the sagittal bucket it was $0.11$ seconds, and in the extreme bucket it was $1.12$ seconds. We note that the computational times are very low, even in the extreme cases,
and thus we consider that for all practical purposes, the running times are
acceptable from the application point of view. The $25$th, $50$th, and $75$th percentiles for ABC as function of Number of Cuts used by the algorithm for the metopic, sagittal, and extreme buckets are presented in Figure~\ref{fig:percentile}. In each bucket we can observe a sharp decline in Area Below the Curve, and our algorithm needing only $3$ cuts to remove over $75$\% of the ABC. This rapid decay suggests that rearrangement is not needed in this application to get good outcomes, providing a principled justification for not using rearrangement in surgerical procedures. We conclude that the algorithm is fast, robust, and does better than current practice in terms of quality of outcome as measured by our objective function.

\begin{figure}
  \centering
  \includegraphics[scale=0.5]{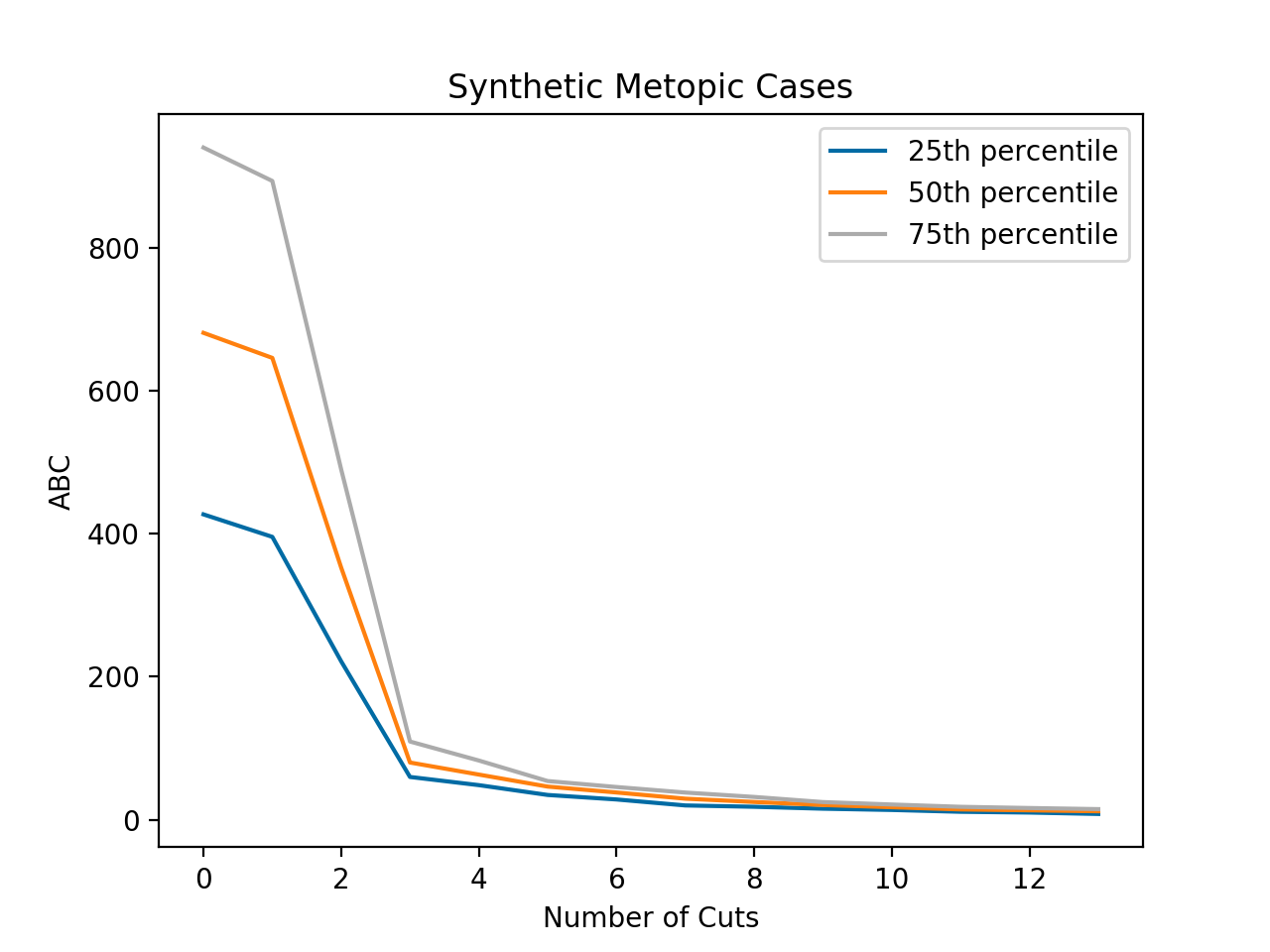}
  \includegraphics[scale=0.5]{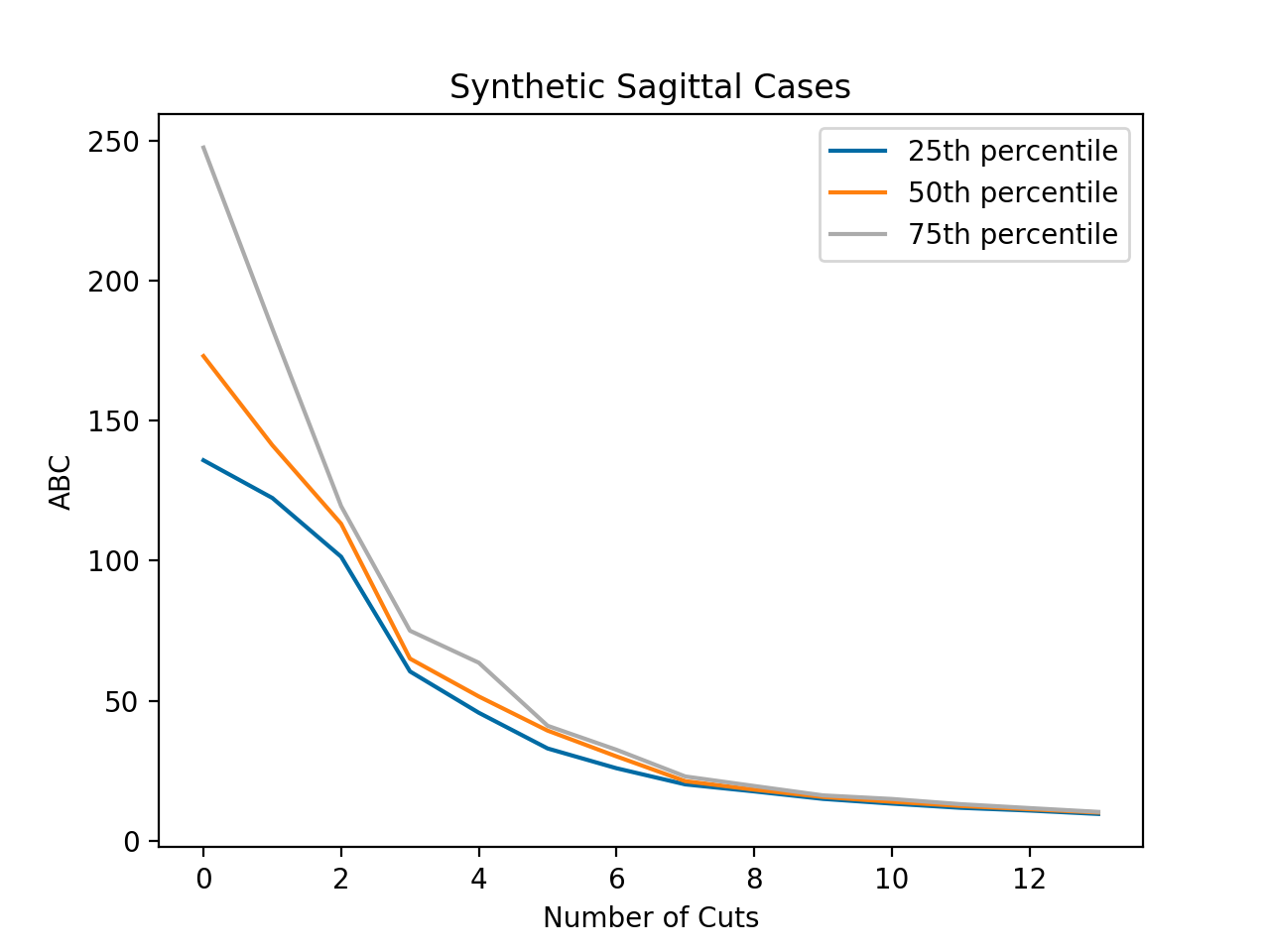}
  \includegraphics[scale=0.5]{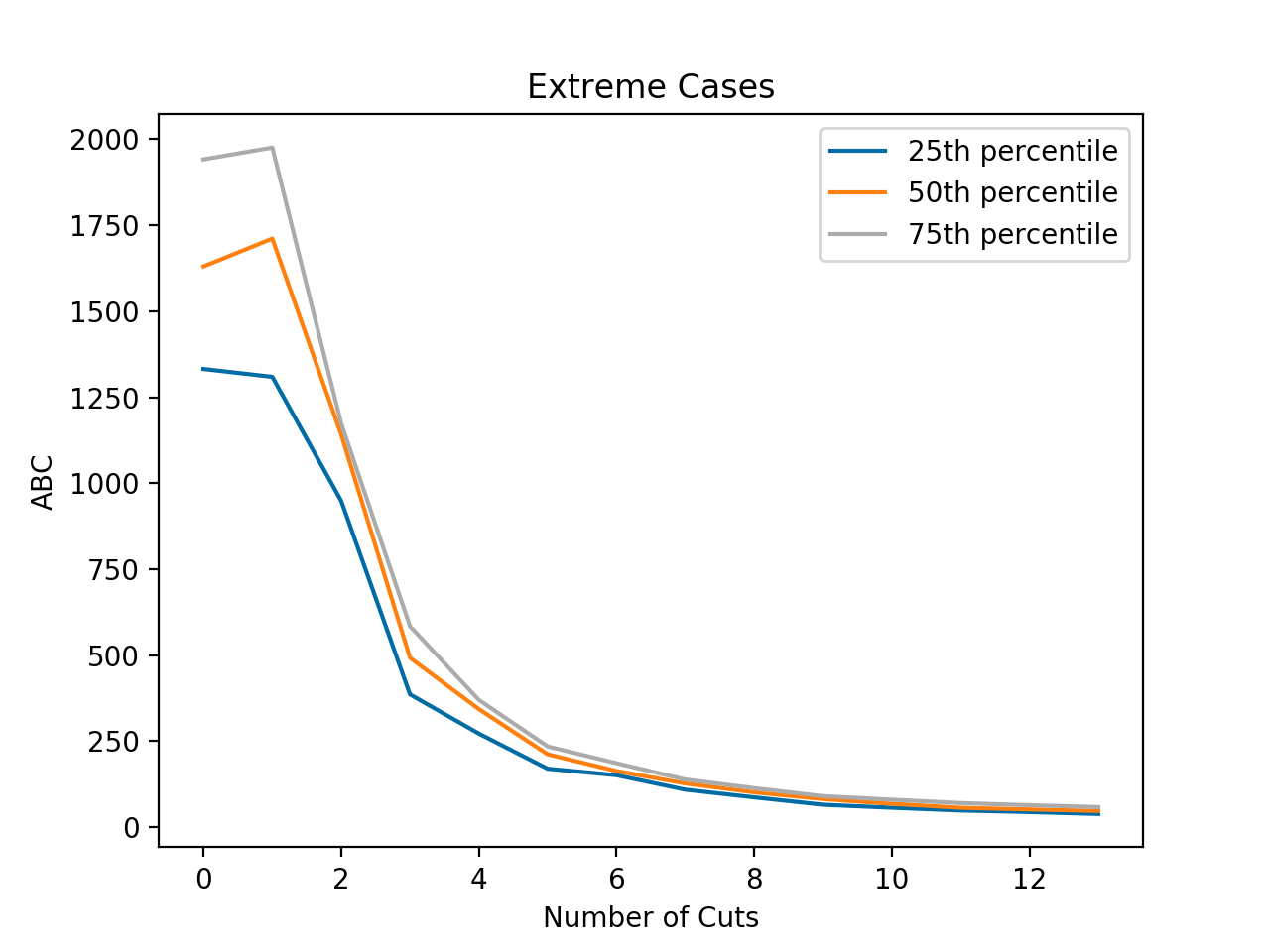}
  \caption{Percentile curves for Synthetic Cases}\label{fig:percentile}
\end{figure}

\section{Hardness in Model with Rearrangement}\label{sec:hardness}
\paragraph{}
The (\ref{prob:cr}) problem with rearrangement involves cutting and rearranging pieces along the length of the bandeau. This problem is of interest as we wish to gain an understanding of the cranial vault remodeling procedure, which involves cutting and rearranging segments of bone over a larger portion of the patient's skull.
We will show that the (\ref{prob:cr}) problem with rearrangement is strongly NP-hard. We will use a reduction from the strongly NP-complete 3-Partition problem~\cite{garey1979computers}.

\paragraph{}
In an instance of 3-Partition the input is a set $A = \{a_1, \dots, a_{3m}\}$ of $3m$ elements for a positive integer $m$. In addition we are given a weight function $s: A \rightarrow \mathbb{Z}_{\geq 0}$ such that $\frac{B}{4} < s(a) < \frac{B}{2}$ for all $a \in A$, such that $\sum_{a \in A} s(a) = m B$ where $B$ is a positive integer. The goal is to decide if there exists a partition of $A$ into $m$ sets $A_{1}, \dots, A_{m}$ such that for each $i \in \{1, \dots, m \}$ we have that $\sum_{a \in A_{i}} s(a) = B$. Note that the choice of $s$, implies that if such a partition exists then $|A_i|=3$ for each $i \in [m]$.

To prove (\ref{prob:cr}) problem with rearrangement is strongly NP-hard by Lemma \ref{lem:psedo_reduction}, we can give a pseudo-polynomial time transformation from 3-Partition.

\begin{lemma}\label{lem:psedo_reduction}\cite{garey1979computers} If $\Pi$ is strongly $NP$-complete and there exists a pseudo-polynomial time transformation from $\Pi$ to $\Pi'$, then $\Pi'$ is strongly $NP$-hard.
\end{lemma}

The idea behind our proof is to take an instance of $3$-Partition and create a deformed curve which consists of a flat line whose discretization segments the line into $3m$ pieces such that the $i$-th piece has length $s(a_i)$, and create an ideal curve which consists of $m$ ``buckets'': length $m$ flat line segments discretized uniformly into unit length segements. The buckets will be separated by tall ``peaks'' which make it impossible for a deformed segment corresponding to a $3$-Partition element to be placed in multiple buckets without incurring some cost. See Figure~\ref{fig:hardness} for an example of this reduction.

\begin{figure}
  \centering
\includegraphics[scale=.75]{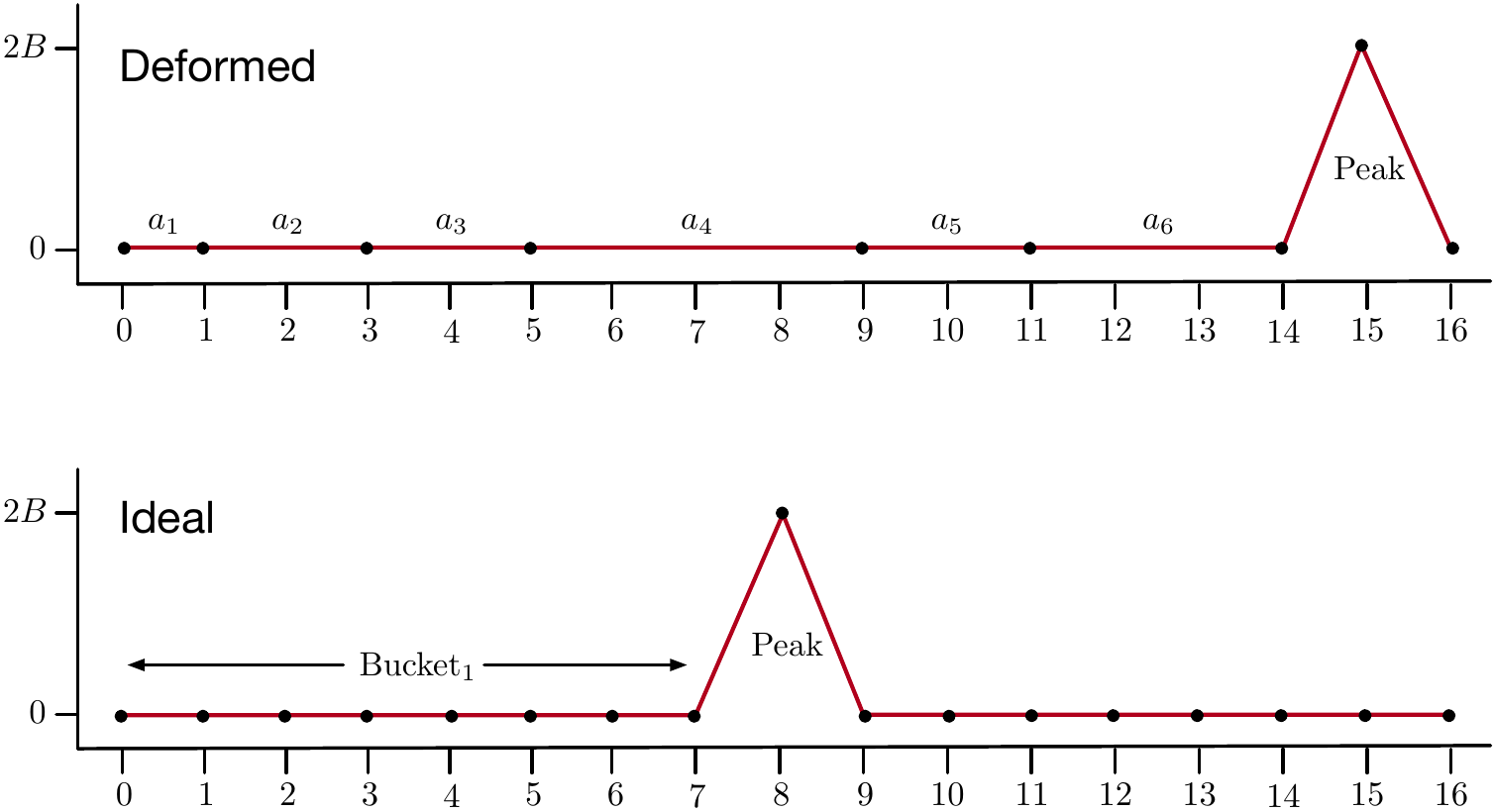}
\caption{The instance of (\ref{prob:cr}) generated from an instance of $3$-Partition with $m=2$, $a_1 = 1$, $a_2 = a_3 = a_5 = 2$, $a_4 =4$, $a_6=3$ by our reduction.  }\label{fig:hardness}
\end{figure}

\begin{theorem}\label{thm:CRPRhard}
  The (\ref{prob:cr}) with rearrangement is strongly NP-hard, even if we restrict to cost functions for which mapping a deformed segment onto an ideal segment incurs $0$ cost if and only if the deformed and ideal segments are equal up to translation and rotation of their endpoints.
\end{theorem}
\begin{proof}
  Consider an instance $I_1$ of $(A,s)$ of $3$-Partition.
  Let $\size{A} = 3m$ and let $B = \frac{1}{m}s(A)$.
  We may assume $m \geq 2$ otherwise $I_1$ is trivial.
  Note that from the definition of $3$-Partition $B$ is a positive integer.

  From $I_1$ we will construct an instance $I_2$ of (\ref{prob:cr}). For $i = 1, \dots, 3m$ let $S_i = \sum_{j=1}^i s(a_j)$ be the prefix sum of the sizes of the first $i$ elements of $A$. Let $S_0 = 0$. The deformed curve in our (\ref{prob:cr}) instance will be a piecewise linear interpolation of the discretization $P \subseteq \R^2$ determined by $3$ sets $P_1, P_2, P_3$ as follows:
  $$ P = P_1 \dot\cup P_2 \dot\cup P_3$$
  where
  $$P_1 = \{(S_i, 0) : i  = 0,1,\dots. 3m\},$$
  $$P_2 = \{(s_{3m} + 2i + 1, 2B): i = 0,\dots, m-2\}, $$
  $$P_3 = \{(s_{3m} + 2i, 0): i = 1,\dots, m-1\}.$$
  Intuitively $f$ is a piecewise linear function which is the constant $0$ function from $0$ to $s(A)$ followed by $m-1$ ``zigzags" between $2B$ and $0$ where each peak and dip are one unit apart. We call these objects peaks of height $2B$.

  The constant $0$ section of $f$ is discretized into $3m$ segments by $P$ where the $i$-th segment is of length $s(a_i)$. We call that segment the segment corresponding to $a_i$. The purpose of the peaks is to fit exactly onto the barriers in the ideal curve which we will describe now.

  The ideal curve $g$ is the piecewise linear interpolation of discretization $Q \subseteq \R^2$ determined by $2$ sets $Q_1, Q_2$ as follows:
  $$Q = Q_1 \dot\cup Q_2$$
  where
  $$ Q_1 = \Z \cap \bigcup_{i = 0}^{m-1} [(B+2)i, (B+2)i + B],$$
  $$ Q_2 = \{((B+2)i + B + 1, 2B): i = 0, \dots, m-2\}.$$
  Intuitively $Q_1$ is a discretization of $m$ constant functions of value $0$ of length $B$ such that they are discretized into $B$ unit length segments. We call each of these segments a \emph{bucket}. $Q_2$ is a set of $m-1$ peaks of height $2B$ which separate the $m$ flat buckets which form $Q_1$.

  Choose any $\delta > 0$ as the uncoverage parameter for $I_2$. Choose any cost function $c$ such that for segments any pair of segments $(a, b) \subseteq P \times Q$ the cost of assigning $a$ to $b$, $c(a,b)$, is $0$ if and only if $a$ and $b$ are equal up to translation and rotation of their endpoints. Choose a number of cuts $k \in \Z$ equal to $(3m+1) +2(m-1) -2 = |P|-2$.

  Observe that our construction can be done in pseudopolynomial time. In particular $P$ and $Q_2$ both consist of $O(m)$ points and $m = O(|A|)$. This is a polynomial number of points. The largest section is $Q_1$ which consists of $O(mB)$ points. Since $mB = s(A)$ this is a pseudopolynomial numbers of points in the size of $I_1$.

  Now we claim that $I_1$ is a Yes instance of $3$-Partition if and only if the optimal value of $I_2$ is $0$.

  First suppose that $I_1$ is a Yes instance. Then $I_1$ has a solution $A_1, \dots, A_m$. Since every point in $P$ is cut by our choice of $k$, we need to assign every line segment of $f$ to a position on $g$. For each $i \in [m]$, order the elements of $A_i$ arbitrarily. Consider the $i$-th bucket on $Q$, i.e. the intersection of $Q_1$ with the interval $[(B+2)i, (B+2)i + B]$. We can arrange the segments of $P$ corresponding to $A_i$ along this interval paying $0$ cost since $\sum_{a \in A_i}s(a) = B$, they are all flat line segments, and $Q_1$ has cut points at every integer point in the interval. To complete our $0$ cost solution to $I_2$, arrange the $m-1$ peaks formed by $P_2 \cup P_3$ onto the $m-1$ peaks formed by cutting $Q_2$. These peaks are of the same height and number so clearly they incur $0$ cost as they match exactly. Hence we have found a $0$ cost solution to $I_2$.

  Now on the other hand, if the optimal value of $I_2$ is $0$ then no part of $Q$ is left uncovered, and every segment of $P$ which is assigned to $Q$ matches exactly. Since the peaks of $Q$ have non-zero height, the optimal solution to $I_2$ will incur positive cost if one endpoint of a segment of $P$ corresponding to an element $a \in A$ is assigned to a point in one bucket and its other endpoint is assigned to a point in a different bucket. Further, by the height of a peak, if a segment corresponding to an element of $A$ is assigned to cover a peak of $Q$ then it will not match exactly and will incur a positive cost. Thus, since every peak of $Q$ is covered, every peak of $P$ must be assigned to a peak of $Q$ in an optimal solution. Since segments of $P$ which corresponding to elements of $A$ cannot cross buckets they are each assigned to a bucket of $Q$. By the lengths of the segments and buckets there can be no overlap, otherwise there will be some uncoverage (which would incur a positive cost). Hence the sum of the lengths of the segments of $P_1$ assigned to the $i$-th bucket of $Q$ is the length of the $i$-th bucket of $Q$, $B$. Construct a solution to $I_1$ by choosing $A_i$ to be the elements corresponding to segments of $P_1$ which are assigned to the $i$-th bucket of $Q$. Therefore $I_1$ is a Yes-instance, as desired.
\end{proof}

\paragraph{}
Our proof of Theorem~\ref{thm:CRPRhard} actually shows that it is strongly NP-complete to decide if an instance of (\ref{prob:cr}) has a $0$ cost solution. Thus we also obtain the following corollary about the hardness of approximation of (\ref{prob:cr}).

\begin{corollary}\label{corollary:approximation} There is no $\nu(|G|, |F|)$-approximation algorithm for the (\ref{prob:cr}) problem unless $\mathcal{P}=NP$, where $\nu$ is any function of $|G|, |F|$ that is computable in polynomial time.
\end{corollary}

\section{Conclusion}
\paragraph{}
The work described in this paper is part of a larger push to improve craniofacial surgical methods through applied mathematics and engineering. The algorithm from Section \ref{sec:dp} is integrated into a pre-operative planning tool, allowing surgeons to pre-plan cut locations. The output of the algorithm and the planning tool is designed to interface with a projection system that shows the cut locations directly on the patient in the operating room.  A prototype system is nearly complete at this stage.

In this paper we presented a formal optimization model, and an algorithm for the craniofacial surgical problem of reshaping the front-orbital bar. We demonstrated its application to several test cases. Our work on the 2D setting continues, and we are currently pursuing extensive clinical testing, and in particular a postoperative, and comparative evaluation of the quality of our solutions.

Future work will focus on 3D, anterior skull craniosynostosis cases,
where surgical incisions and optimization problems are much more
complex.

\bibliographystyle{splncs04}
\bibliography{sickkids}

\end{document}